\documentclass{amsart}
\usepackage{color}
\usepackage{graphicx}
\usepackage{amsmath,amscd}
\usepackage{amssymb}
\input epsf
\usepackage{epsfig}

\usepackage{enumerate}
\newtheorem{theorem}{Theorem}

\newtheorem{proposition}[theorem]{Proposition}
\theoremstyle{definition}

\newtheorem{example}[theorem]{Example}

\theoremstyle{remark}

\numberwithin{equation}{section}

\newcommand{\R}{{\mathbb R}}

\newcommand {\hide}[1]{}

\begin{document}
    \title[On a tropical version of the Jacobian conjecture]
    {On a tropical version of the Jacobian conjecture}
    {
        \let\thefootnote\relax\footnote{2010 Mathematics Subject Classification 14T05}
    }
    \author{Dima Grigoriev}
    \address{CNRS, Math\'ematiques, Universit\'e de Lille, Villeneuve d'Ascq, 59655, France}
    \email{dmitry.grigoryev@univ-lille.fr}
    \author{Danylo Radchenko}
    \address{Max Planck Institute for Mathematics, Bonn, Germany}
    \email{danradchenko@gmail.com}
    
    \begin{abstract}
        We prove for a tropical rational map that if for any point the convex hull of Jacobian matrices at smooth points in a neighborhood of the point does not contain singular matrices then the map is an isomorphism. We also show that a tropical polynomial map on the plane is an isomorphism if all the Jacobians have the same sign (positive or negative). In addition, for a tropical rational map we prove that if the Jacobians have the same sign and if its preimage is a singleton at least at one regular point then the map is an isomorphism.
    \end{abstract}
    
    \maketitle
    
    \section*{Introduction}
    
    We provide two criteria for a tropical map of $n$-dimensional real spaces to be an isomorphism. The basic concepts of tropical (or min-plus) mathematics one can find in \cite{MS}. The first criterion concerns tropical rational maps and states that if at any point the convex hull of Jacobian matrices at smooth points in a neighborhood of the point does not contain singular matrices then the map is an isomorphism (Proposition~\ref{local} and Theorem~\ref{global}). Note that there is totally a finite number of Jacobian matrices corresponding to polyhedra on which the tropical rational map is linear. We also show that this criterion is not necessary (even for 2-dimensional tropical polynomial isomorphisms).
    
    The second criterion deals with tropical polynomial maps and holds for 2-dimensional real spaces. It claims that if the Jacobians at all the smooth points have the same sign (either positive or negative) then the tropical polynomial map is an isomorphism (Theorem~\ref{plane}). We construct an example which demonstrates that this criterion fails for tropical rational 2-dimensional maps as well as for tropical polynomial 3-dimensional ones. In addition, for a tropical rational map we prove that if the Jacobians at all the smooth points have the same sign and if its preimage is a singleton at least at one regular point then the map is an isomorphism (Theorem~\ref{general}).
    
    \section*{Weak version of the tropical Jacobian conjecture}
    By a tropical algebraic rational function (in $n$ variables) we mean a function from $\R^n$ to $\R$ of the form $\min\{A_1,\dots, A_k\}-\min\{B_1,\dots, B_l\}$ where $A_1,\dots,A_k$, $B_1,\dots,B_l$ are linear functions with rational coefficients (when the coefficients at the variables are positive integers one talks about tropical rational functions). They play the role of a tropical analog of algebraic functions (cf. \cite{G}). In fact, one could stick with more generally, real coefficients.
    
    From a different point of view, defining a tropical algebraic rational 
    function amounts to giving a finite decomposition $\R^n=\bigcup_i C_i$,
    where each $C_i$ is a closed $n$-dimensional polyhedron and the interiors 
    of $C_i$ and $C_j$ do not intersect if $i\ne j$, and specifying a 
    linear function $f_i$ on each $C_i$ in such a way that $f_i$ and $f_j$ 
    agree on $C_i\cap C_j$. Similar description applies in general to 
    tropical algebraic rational mappings from $\R^n$ to $\R^m$. We will call 
    $(f_i,C_i)$ or just $f_i$ the linear pieces of $f$.
    
    Observe that if a tropical algebraic rational map $\R^n \to \R^n$ is a bijection then its inverse is also a tropical algebraic rational map. In this case we call such a map an isomorphism.
    
    \begin{proposition}\label{local}
        If a tropical algebraic rational map $f\colon\R^n\to\R^n$
        is locally injective, then it is an isomorphism. 
    \end{proposition}
    \begin{proof}
        Let $f_1,\dots,f_N$ be the linear pieces of $f$.
        Since $f$ is locally injective, each $f_i$ has to be invertible, and hence
        for any compact $K\subset \R^n$ we have that $f_i^{-1}(K)$ is compact.
        Therefore, $f^{-1}(K) \subseteq \bigcup_{i=1}^{N}f_i^{-1}(K)$ is compact,
        and thus $f$ is a proper mapping. By the domain 
        invariance theorem~\cite{B} the mapping~$f$ is locally a homeomorphism. 
        The result then follows directly from~\cite[Th.~2]{H} that, specialized to 
        our setting, says that if a continuous mapping $f\colon\R^n\to\R^n$ is a 
        local homeomorphism and is proper, then it is a global homeomorphism.
    \end{proof}
    This result can be treated as a weak tropical version of the 
    Jacobian conjecture in the sense that for tropical rational
    algebraic maps local invertibility implies global invertibility.
    The problem remains how to tell if a given map is locally
    an isomorphism.
    There is a simple sufficient condition for local 
    invertibility of Lipschitz mappings that can be easily stated 
    in our case (it is not hard to show 
    that tropical rational algebraic maps are Lipschitz).
    Let $f\colon\R^n\to\R^n$ be a tropical algebraic rational mapping with 
    linear pieces $(f_i,C_i)$, $i=1\dots,N$.
    For any point $x\in\R^n$ let $I=\{i\,|\, x\in C_i\}$. 
    We then define $\partial_f(x)$ to be the convex hull 
    of the set of differentials (Jacobian matrices)
    of $f_i$, $i\in I$ (in the space of matrices ${\rm Mat}_{n\times n}(\R)$). 
    The following is a special case of the inverse function theorem for Lipschitz 
    mappings proved by Clarke~\cite[Th.~1]{C}.
    \begin{theorem}\label{global}
        If $f\colon\R^n\to\R^n$ is a tropical algebraic rational mapping 
        and the set $\partial_f(x_0)$ contains only nonsingular 
        matrices, then $f$ is a homeomorphism in a neighborhood of $x_0$.
    \end{theorem}
    
    Note, however, that the sufficient condition in 
    Theorem~\ref{global} for $f$ to be an isomorphism 
    is not necessary, as the following example shows.
    
    \begin{example}
        We construct a tropical polynomial isomorphism 
        $f\colon \R^2\to \R^2$ as a composition of two tropical 
        polynomial maps (isomorphisms): a lower-triangular map
        $$(x,y)\mapsto (x,\, y+\min\{\alpha x,\, \beta x\}),\, \alpha < \beta$$
        \noindent and an upper-triangular one:
        $$(x,y)\mapsto (x+\min\{ay,\, by\},\, y),\, a<b.$$
        
        Then $f(x,y)=$
        $$f_1:=(x+a(y+\alpha x),\, y+\alpha x) \quad \mbox{if} \quad x>0,\, y+\alpha x >0;$$
        $$f_2:=(x+b(y+\alpha x),\, y+\alpha x) \quad \mbox{if} \quad x>0,\, y+\alpha x <0;$$
        $$f_3:=(x+a(y+\beta x),\, y+\beta x) \quad \mbox{if} \quad x<0,\, y+\beta x >0;$$
        $$f_4:=(x+b(y+\beta x),\, y+\beta x) \quad \mbox{if} \quad x<0,\, y+\beta x <0.$$
        
        The corresponding matrices of differentials at the origin are as follows:
        $$\left(\begin{array}{cc} 1+a\alpha & a \\ \alpha & 1 \end{array} \right), \, \left(\begin{array}{cc} 1+b\alpha & b \\ \alpha & 1 \end{array} \right),\, \left(\begin{array}{cc} 1+a\beta & a \\ \beta & 1 \end{array} \right),\, \left(\begin{array}{cc} 1+b\beta & b \\ \beta & 1 \end{array} \right).$$
        \noindent The sum of the second and the third matrices is singular when $(\beta-\alpha)(b-a)=4$ (in particular, one can put $\beta=b=2,\, \alpha=a=0$). Thus, $\partial_f(0,0)$ contains a singular matrix.
    \end{example}
    
    
    \section*{Strong version of the tropical Jacobian conjecture}
    Denote by $J_i,\, 1\le i\le N$ the determinants (Jacobians) 
    of the differentials of~$f_i$. 
    If~$f$ is an isomorphism then all $J_i$ must have the same sign 
    (either positive or negative). This follows from a general result about 
    homeomorphisms (see, for example,~\cite[Th.~5.22]{HK}) 
    or can be verified directly by comparing the Jacobians
    $J_i,\, J_j$ for a pair of $f_i,\, f_j$ defined on adjacent 
    ($n$-dimensional) polyhedra separated by a common $(n-1)$-dimensional facet.
    Observe that the condition on $\partial _f(p)$ in 
    Theorem~\ref{global} implies that all the Jacobians 
    have the same sign since for a pair of real matrices $A,\, B$ 
    with $\det(AB)<0$ there exists $0<t<1$ such that $\det(tA+(1-t)B)=0$.
    
    The condition that all $J_i$ are nonzero and have the same sign 
    is not sufficient to guarantee that $f$ is 
    an isomorphism (see Example~\ref{space} below),
    but we will now show that it is almost sufficient.
    In analogy to the case of smooth maps, we call $y\in\R^n$
    a regular value of $f$, if it does not belong to the set
    $\bigcup_{i=1}^N f(\partial C_i)$. A version of Sard's theorem
    is true (and is trivial) in this case: 
    the set of regular values is dense in $\R^n$.
    \begin{theorem}\label{general}
        A necessary and sufficient condition for a tropical algebraic 
        rational mapping $f\colon\R^n\to\R^n$ to be an isomorphism
        is that all Jacobians $J_i$ have the same sign, and 
        that $|f^{-1}(y_0)|=1$ for at least one regular value $y_0$.
    \end{theorem}
    \begin{proof}
        The necessity is clear. To prove sufficiency we will use 
        standard techniques from topological degree theory
        (for a general reference, see~\cite[Ch.~IV]{OR}). 
        
        Without loss of generality we can assume that all 
        Jacobians are positive. 
        Since $f$ is proper, it has a well-defined degree $\deg(f)$
        (as a mapping from the one-point compactification of $\R^n$
        to itself), and by computing it using the preimage of the 
        regular value $y_0$ we get $\deg(f)=1$. 
        This, in turn, implies that $|f^{-1}(y)|=1$ holds for any 
        regular value $y$. The proof would then be complete if we could 
        show that $f$ is an open mapping: indeed, if $f(x_1)=f(x_2)$
        for some $x_1\ne x_2$, then $|f^{-1}(y)|\ge 2$ for all $y$ 
        in a sufficiently small neighborhood of $f(x_1)$, which would 
        contradict the above computation for regular~$y$.
        
        To show that $f$ is open we use an argument similar to the one used 
        in \cite[Proof of Thm. 3]{R}.
        Note that $|f^{-1}(y)|\le N$ for all $y$, where $N$ is the
        number of linear pieces of~$f$. Therefore, for any $x\in\R^n$
        we can find a small neighborhood $U$ of $x$ such that
        $f^{-1}(f(x))\cap U = \{x\}$. Let $V$ be a small open ball
        around $f(x)$ that does not intersect~$f(\partial U)$.
        Then for some smooth point $x_0\in U\cap f^{-1}(V)$ we have
        that $f(x_0)$ is a regular value, and hence 
        $\deg(f,U,f(x_0))=1$. Since $\deg(f,U,p)$ only depends on
        the connected component of $\R^n\smallsetminus f(\partial U)$ 
        to which $p$ belongs, we get that $\deg(f,U,y)=1$ for 
        all $y\in V$. But if $\deg(f,U,y)\ne 0$, then $y\in f(U)$
        (see~\cite[Cor.~IV.2.5(3)]{OR}), so that $V\subset f(U)$.
        Thus $f$ is an open mapping.
    \end{proof}
    Let us remark that this result gives a simple computational
    approach to checking that a given tropical rational algebraic map $f$ 
    is an isomorphism. Indeed, to check the above conditions one only needs 
    to compute the set of linear functions $f_i$ and the preimage 
    of a generic point~$y_0$, but both of these tasks are easy 
    to do using linear programming.    
    
    It turns out that if $n=2$ and $f$ is a tropical algebraic mapping, 
    so that each coordinate of $f$ is concave and piecewise linear,
    then a stronger result is true.   
    \begin{theorem}\label{plane}
        If $f\colon\R^2\to\R^2$ is a tropical algebraic mapping 
        (i.e., $f=(\phi_1,\, \phi_2)$ where $\phi_i$
        are tropical algebraic functions) such that all the 
        Jacobians $J_i$ have the same sign, then $f$ is an isomorphism.
    \end{theorem}
    
    \begin{proof} 
    For $a\in \R$ such that the preimage $\phi_1^{-1}(a)\subset \R^2$ 
    is non-empty, it is the boundary of the convex 
    polygon $P:=\{x\, |\, \phi_1(x)\ge a\}$. Consider the 
    supporting to $P$ lines $L_1,\dots,L_m$. Either $m=2$ 
    and $L_1,\, L_2$ are parallel or one can renumber 
    $L_1,\dots,L_m$ to make their slopes in the 
    clock-wise order (with respect to an arbitrary point of $P$). 
    
    In the latter case $\phi_1^{-1}(a)$ is a connected polygonal 
    line with the consecutive edges (some of them, perhaps, 
    unbounded) lying on $L_1,\dots,L_m$, respectively. 
    Due to the condition imposed on the Jacobians $J_i$ 
    in Theorem~\ref{plane}, $\phi_2$ is strictly monotone 
    along $\phi_1^{-1}(a)$. Therefore, $\phi_1^{-1}(a)$ 
    has to be unbounded with just two unbounded edges 
    $L_1,\, L_m$, and $\phi_2$ restricted to $\phi_1^{-1}(a)$ 
    (being a piecewise linear function with a finite number
    of pieces) provides an isomorphism with $\R$. Thus, 
    we conclude that $f$ is an isomorphism if $\phi_1^{-1}(a)$ 
    is connected (or empty) for all $a\in \R$ by
    Proposition~\ref{local}. By symmetry, $f$ is also 
    an isomorphism if $\phi_2^{-1}(a)$ is connected 
    (or empty) for all $a\in \R$.
    
    Now we study the case when $m=2$ and $L_1,\, L_2$ are parallel. 
    Let $\phi_i=:\min\{A_{i,1},\dots, A_{i,s_i}\}$, $1\le i\le 2$.
    Then each linear function $A_{1,q},\, 1\le q\le s_1$ is 
    constant on $L_1$. Hence the graph $\Phi_1\subset \R^3$ 
    of $\phi_1$ is the Cartesian product of $L_1$ and a 
    convex polygonal line $Q\subset \R^2$ which can be 
    treated as the graph of a univariate concave piecewise 
    linear function. The latter function is bounded from 
    above since~$\Phi_1$ contains $L_1,\, L_2$. 
    Therefore, $\phi_1$ attains its maximum on a 
    line $L$ parallel to $L_1,\, L_2$. 
    
    Thus, it remains to consider the case when for 
    both $\phi_1,\, \phi_2$ there exist $a_1,\, a_2\in \R$ 
    such that $\phi_1^{-1}(a_1),\, \phi_2^{-1}(a_2)$ 
    are disconnected, so they consist of pairs of 
    parallel lines. So, $L$ separates two adjacent 
    pieces on which $\phi_1$ is linear with their differentials 
    being collinear with a negative coefficient of collinearity. 
    If at some point $y$ of $L$ the piecewise linear 
    function $\phi_2$ is non-singular (thereby, is 
    linear in a neighborhood of $y$) then (the only) 
    two linear maps $f_i,\, f_j$ defining $f$ in a 
    neighborhood of $y$ have the Jacobians $J_i,\, J_j$ 
    with opposite signs contrary to the condition of 
    Theorem~\ref{plane}.
    
    Otherwise, if $\phi_2$ has an edge on $L$  separating 
    two pieces on which $\phi_2$ is linear, then for 
    two linear maps $f_i,\, f_j$ which define $f$ in 
    a neighborhood of any point of this edge their 
    Jacobians $J_i,\, J_j$ both vanish.
    Obviously, $\phi_1$ can't attain its maximum on 
    a $2$-dimensional facet since the differential 
    of $\phi_1$ would vanish on it. We get a contradiction 
    with the supposition that there exist $a_1,\, a_2\in \R$ 
    such that $\phi_1^{-1}(a_1),\, \phi_2^{-1}(a_2)$ are disconnected.
    \end{proof}
    The simple condition of Theorem~\ref{plane}
    is no longer sufficient for $n\ge 3$, as the following example shows.
    \begin{example}
        \label{space}
        We will construct a tropical polynomial mapping $f\colon\R^n\to\R^n$
        , $n\ge 3$, for which the Jacobians $J_i > 0$ for all $i$, 
        but which is not an isomorphism.
        
        Clearly, it is enough to construct such an example for $n=3$.
        We start with a two-dimensional tropical rational 
        mapping $g\colon\R^2\to\R^2$ defined by
        \[(x,y) \quad\mapsto\quad (|x|-|y|,\; |x+y|-|x-y|).\]
        Note that the Jacobian of each linear piece of $g$ is equal to $2$,
        but $g(x,y)=g(-x,-y)$, so that $g$ is not injective. To construct
        $f$ we start with the mapping that sends $(x,y,z)$ to $(g(x,y),z)$
        and then apply elementary piecewise linear transformations to get rid 
        of all the subtractions in the min-plus form. 
        More precisely, let us define $h\colon\R^3\to\R^3$~by
        \[(x,y,z) \quad\mapsto\quad (-|y|-|x+y|+z,\; -|x|-|x-y|+z,\; -|x+y|-|x|+z).\]
        One can check that $h$ is a tropical Laurent polynomial map 
        (since $-|x|=\min(-x,x)$ is), 
        its Jacobian is equal to $2$ at every smooth point, 
        and $h(x,y,z)=h(-x,-y,z)$. To get the mapping $f$, we 
        compose $h$ with $z\mapsto z+2x+2y$ to obtain
        \begin{align*}
        f(x,y,z) = \;
        &(\min(2y,0)+\min(2x+2y,0)+z+x,\\
        &\min(2x,0)+\min(2x,2y)+z+y,\\
        &\min(2x,0)+\min(2x+2y,0)+z+y).
        \end{align*}
        Then $f$ is a tropical polynomial map whose Jacobian
        at every smooth point is equal to $2$, but it is not an isomorphism
        since $f(-x,-y,z+4x+4y)=f(x,y,z)$.
    \end{example}
    
    \vspace{2mm}
    
    {\bf Acknowledgements.} The first author is grateful to ANR/DFG grant 
    
    \noindent SYMBIONT. The second author is grateful to Max Planck Institute
    for Mathematics in Bonn for its hospitality and financial support.

\end{document}